\documentclass[a4paper, reqno, 11pt]{amsart}
\pdfoutput=1
\newif\ifprivate
\privatefalse

\usepackage{a4wide}

\usepackage[utf8]{inputenc}
\usepackage[T1]{fontenc}
\usepackage{lmodern}
\usepackage[draft=false,kerning=true]{microtype}

\usepackage{enumerate}
\usepackage{mathtools}
\usepackage{textcomp}
\usepackage{tikz-cd}
\usetikzlibrary{backgrounds,scopes}
\usetikzlibrary{automata}

\usepackage{hyperref}
\usepackage{todonotes}

\ifprivate
\usepackage[mark]{gitinfo2}
\renewcommand{\gitMark}{\jobname\,\textbullet{}\,\gitFirstTagDescribe\,\textbullet{}\,\gitAuthorName,\,\gitAuthorIsoDate}
\newcommand{\TODO}[1]%
{\par\fbox{\begin{minipage}{0.9\linewidth}\textbf{TODO:} #1\end{minipage}}\par}
\fi

\newcommand{\proofparagraph}[1]{\medskip\par\noindent{\itshape #1.} }

\DeclarePairedDelimiter{\abs}{\lvert}{\rvert}

\newcommand{\C}{\mathbb{C}}

\newcommand{\calD}{\mathcal{D}}

\newcommand{\calV}{\mathcal{V}}
\newcommand{\calX}{\mathcal{X}}

\DeclarePairedDelimiter{\ceil}{\lceil}{\rceil}
\newcommand{\dd}{\mathrm{d}}

\DeclarePairedDelimiterXPP{\f}[2]{\foperator{#1}}(){}{#2}
\DeclarePairedDelimiterXPP{\fexp}[1]{\exp}(){}{#1}
\DeclarePairedDelimiter{\floor}{\lfloor}{\rfloor}
\newcommand{\foperator}[1]{\mathop{{#1}\empty{}}}
\DeclarePairedDelimiter{\fractional}{\{}{\}}
\DeclarePairedDelimiterXPP{\inftynorm}[1]{}{\lVert}{\rVert}{_\infty}{#1}

\DeclarePairedDelimiter{\iverson}{[}{]}

\renewcommand{\MR}[1]{}
\newcommand{\N}{\mathbb{N}}
\DeclarePairedDelimiter{\norm}{\lVert}{\rVert}
\DeclarePairedDelimiterXPP{\Oh}[1]{\foperator{O}}(){}{#1}
\DeclarePairedDelimiterXPP{\oh}[1]{\foperator{o}}(){}{#1}

\newcommand{\R}{\mathbb{R}}

\DeclarePairedDelimiterXPP{\Res}[2]{\operatorname{Res}}(){}{#1, #2}
\DeclarePairedDelimiter{\set}{\{}{\}}
\DeclarePairedDelimiterX{\setm}[2]{\{}{\}}{#1 \colon \mathopen{}#2}

\newcommand{\tpmod}[1]{\ensuremath{\undisp{\pmod{#1}}}}

\makeatletter\newcommand\undisp[1]{\bgroup\@displayfalse #1\egroup}\makeatother

\newcommand{\Z}{\mathbb{Z}}

\newtheorem{theorem}{Theorem}

\newtheorem{lemma}{Lemma}[section]
\newtheorem{proposition}[lemma]{Proposition}

\theoremstyle{remark}

\newtheorem{remark}[lemma]{Remark}

\numberwithin{equation}{section}
\numberwithin{figure}{section}
\numberwithin{table}{section}

\begin{document}
\title[Analysis of Summatory Functions of Regular Sequences]{Esthetic
  Numbers and Lifting Restrictions on the Analysis of Summatory
  Functions of Regular Sequences}

\author[C.~Heuberger]{Clemens Heuberger}
\address{Clemens Heuberger,
  Institut f\"ur Mathematik, Alpen-Adria-Universit\"at Klagenfurt,
  Universit\"atsstra\ss e 65--67, 9020 Klagenfurt am W\"orthersee, Austria}
\email{\href{mailto:clemens.heuberger@aau.at}{clemens.heuberger@aau.at}}

\author[D.~Krenn]{Daniel Krenn}
\address{Daniel Krenn,
  Institut f\"ur Mathematik, Alpen-Adria-Universit\"at Klagenfurt,
  Universit\"atsstra\ss e 65--67, 9020 Klagenfurt am W\"orthersee, Austria}
\email{\href{mailto:math@danielkrenn.at}{math@danielkrenn.at} \textit{or}
  \href{mailto:daniel.krenn@aau.at}{daniel.krenn@aau.at}}

\thanks{C.~Heuberger and D.~Krenn are supported by the
   Austrian Science Fund (FWF): P\,28466-N35.}



\begin{abstract}
  When asymptotically analysing the summatory function of a
  $q$-regular sequence in the sense of Allouche and Shallit, the
  eigenvalues of the sum of matrices of the linear representation of
  the sequence determine the ``shape'' (in particular the growth) of
  the asymptotic formula. Existing general results for determining the
  precise behavior (including the Fourier coefficients of the
  appearing fluctuations) have previously been restricted
  by a technical condition on these eigenvalues.

  The aim of this work is to lift these restrictions by providing an
  insightful proof based on generating functions for the main pseudo
  Tauberian theorem for all cases simultaneously. (This theorem is the
  key ingredient for overcoming convergence problems in Mellin--Perron
  summation in the asymptotic analysis.)

  One example is discussed in more detail: A precise asymptotic
  formula for the amount of esthetic numbers in the first~$N$ natural
  numbers is presented. Prior to this only the asymptotic amount of
  these numbers with a given digit-length was known.
\end{abstract}

\maketitle

\thispagestyle{empty}
\setcounter{page}{0}

\clearpage

\section{Introduction}

This extended abstract studies the asymptotic behaviour of summatory functions
of $q$-regular sequences. We start with a definition of $q$-regular sequences.

\subsection{\texorpdfstring{$q$}{q}-Regular Sequences}\label{section:introduction:regular-sequences}
An introduction and formal definition of $q$-regular sequences (via
the so-called $q$-kernel) is given by Allouche and
Shallit~\cite{Allouche-Shallit:1992:regular-sequences} and
\cite[Chapter~16]{Allouche-Shallit:2003:autom}. We settle here for an
equivalent formulation which is the most useful for our
considerations.

Let $q\ge 2$ be a
fixed integer and $(x(n))_{n\ge 0}$ be a sequence.
Then
$(x(n))_{n\ge 0}$ is $q$-regular if and only if there exists a vector valued
sequence $(v(n))_{n\ge 0}$ whose first component coincides with
$(x(n))_{n\ge 0}$ and there exist square matrices $A_0$, \ldots, $A_{q-1}\in\C^{d\times d}$ such that
\begin{equation}\label{eq:linear-representation}
  v(qn+r) = A_r v(n)\qquad\text{for $0\le r<q$, $n\ge 0$;}
\end{equation}
see Allouche and
Shallit~\cite[Theorem~2.2]{Allouche-Shallit:1992:regular-sequences}.
This is called a \emph{$q$-linear representation} of $x(n)$.

We note that a linear representation~\eqref{eq:linear-representation}
immediately leads to an explicit expression for $x(n)$ by induction:
Let $r_{\ell-1}\ldots r_0$ be the $q$-ary digit expansion of $n$. Then
\begin{equation*}
  x(n) = e_1 A_{r_0}\dotsm A_{r_{\ell-1}}v(0)
\end{equation*}
where $e_1=\begin{pmatrix}1& 0& \dotsc& 0\end{pmatrix}$.

Regular sequences are related to
divide-and-conquer algorithms, therefore they have been intensively
investigated in the literature in many particular cases; see, for example,
\cite{Drmota-Szpankowski:2013:divide-and-conquer},
\cite{Dumas:2013:joint},
\cite{Dumas:2014:asymp},
\cite{Dumas-Lipmaa-Wallen:2007:asymp}
\cite{Grabner-Heuberger:2006:Number-Optimal},
\cite{Grabner-Heuberger-Prodinger:2005:counting-optimal-joint},
\cite{Heuberger-Krenn-Prodinger:2018:pascal-rhombus},
\cite{Heuberger-Kropf-Prodinger:2015:output} and
\cite{Hwang-Janson-Tsai:2017:divide-conquer-half}
for a more detailed overview.
The best-known example for a $2$-regular function is the binary sum-of-digits
function.

\subsection{Summatory Functions}

Of particular
interest is the analysis of the summatory function (i.e., the sequence
of partial sums) of a regular sequence, not least because of its
relation to the expectation of a random element of the sequence (with
respect to uniform distribution on the nonnegative integers smaller
than a certain $N$).
In \cite{Heuberger-Krenn-Prodinger:2018:pascal-rhombus}, Prodinger and
the two authors of this extended abstract provide a theorem
decomposing the summatory function into periodic fluctuations
multiplied by some scaling functions; the Fourier coefficients of
these periodic fluctuations are provided as well. Although this result
is quite general, the proof in
\cite{Heuberger-Krenn-Prodinger:2018:pascal-rhombus} imposes a
restriction on the asymptotic growth. One major aim of this work is to
lift this restriction by completely getting rid of the corresponding
technical condition. We formulate the full main theorem in Section~\ref{sec:asy-summatory} and
the theorem stating the underlying pseudo-Tauberian
argument in Section~\ref{sec:pseudo-tauber}.

\subsection{The Proof}

The proof of the extended pseudo-Tauberian theorem contained in this
extended abstract not only covers the previously excluded cases, but also works
for the existing theorem in
\cite{Heuberger-Krenn-Prodinger:2018:pascal-rhombus}. In particular
the proof of the main result does not need a case distinction, but the
contained proof supersedes the existing one.
(Besides, it also is much shorter.) This is reached by
changing the perspective to a more general point of view; we use a
generating functions approach. Beside proving the theorem, this also
gives additional insights. For example, the cancellations in the proof
in \cite{Heuberger-Krenn-Prodinger:2018:pascal-rhombus} seem to be a
kind of magic at that point, but with the new approach, it is now
clear and no surprise anymore that they have to appear.

\subsection{Esthetic Numbers}

A further main contribution of this extended abstract is the
precise asymptotic analysis of $q$-esthetic numbers, see~De~Koninck
and Doyon~\cite{Koninck-Doyon:2009:esthetic-numbers}. These are
numbers whose $q$-ary digit expansion satisfies the condition that
neighboring digits differ by exactly one. The sequence of such numbers
turns out to be $q$-automatic, thus are $q$-regular and can also be
seen as an output sum of a transducer; see the first author's joint
work with Kropf and
Prodinger~\cite{Heuberger-Kropf-Prodinger:2015:output}. However, the
asymptotics obtained by using the main result of
\cite{Heuberger-Kropf-Prodinger:2015:output}---in fact, this result is
recovered as a corollary of the main result of
\cite{Heuberger-Krenn-Prodinger:2018:pascal-rhombus}---is degenerated
in the sense that the provided main term and second order term both
equal zero. On the other hand, using a more direct approach via our
main theorem brings up the actual main term and the fluctuation in
this main term. The full theorem is formulated in
Section~\ref{sec:esthetic-numbers}.
Prior to this precise analysis,
the authors of~\cite{Koninck-Doyon:2009:esthetic-numbers} only performed an analysis
of esthetic numbers by digit-length (and not by the number itself).

The approach used in the analysis of $q$-esthetic numbers can easily
be adapted to numbers defined by other conditions on the word of
digits of their $q$-ary expansion.

\subsection{Dependence on Residue Classes}

The analysis of $q$-esthetic numbers also brings another aspect into
the light of day, namely a quite interesting dependence of the
behaviour with respect to~$q$ on different moduli:
\begin{itemize}
\item The dimensions in the matrix approach of
  \cite{Koninck-Doyon:2009:esthetic-numbers} need to be increased for
  certain residue classes of~$q$ modulo~$4$ in order to get a
  formulation as a $q$-automatic and $q$-regular sequence,
  respectively.
\item The main result in~\cite{Koninck-Doyon:2009:esthetic-numbers}
  already depends on the parity of $q$ (i.e., on $q$
  modulo~$2$). This reflects our Theorem~\ref{theorem:esthetic:asy}
  by having $2$-periodic
  fluctuations (in contrast to $1$-periodic fluctuations in the main
  Theorem~\ref{theorem:simple}).
\item Surprisingly, the error term in the resulting formula of
  Theorem~\ref{theorem:esthetic:asy} depends on the residue class of $q$ modulo~$3$. This is
  due to the appearance of an eigenvalue~$1$ in certain cases.
\item As an interesting side-note: In the same (up to this point not
  specified; see below) spectrum, the algebraic multiplicity of the
  eigenvalue~$0$ changes again only modulo~$2$.
\end{itemize}
The spectrum above consists of the eigenvalues of the sum of matrices
of the $q$-linear representation of the sequence.

\subsection{Symmetrically Arranged Eigenvalues}

The second of the four bullet points above comes from a particular
configuration in the spectrum. Whenever eigenvalues are arranged as
vertices of a regular polygon, then their influence can be collected;
this results in periodic fluctuations with larger period than~$1$.
We elaborate on the influence of such eigenvalues in
Section~\ref{sec:symmetric}.
This is then used in the particular case of esthetic
numbers, but might also be used in conjunction with the output sum
of transducers; to be precise, for obtaining the second order term in
the main result of~\cite{Heuberger-Kropf-Prodinger:2015:output}.

\section{Esthetic Numbers}
\label{sec:esthetic-numbers}
Let again be $q\geq2$ a fixed integer. We call a nonnegative integer~$n$ a
\emph{$q$-esthetic number} (or simply an \emph{esthetic number}) if its
$q$-ary digit expansion $r_{\ell-1} \dots r_0$ satisfies
$\abs{r_j - r_{j-1}} = 1$ for all $j\in\set{1,\dots,\ell-1}$;
see~De~Koninck and Doyon~\cite{Koninck-Doyon:2009:esthetic-numbers}.

In~\cite{Koninck-Doyon:2009:esthetic-numbers} the authors count
$q$-esthetic numbers with a given length of their $q$-ary digit
expansion. They provide an exact as well as an asymptotic formula for
these counts. We aim for a more precise analysis and head for an
asymptotic description of the amount of $q$-esthetic numbers up the an
arbitrary value~$N$ (in contrast to only powers of~$q$
in~\cite{Koninck-Doyon:2009:esthetic-numbers}).

\subsection{A $q$-Linear Representation}

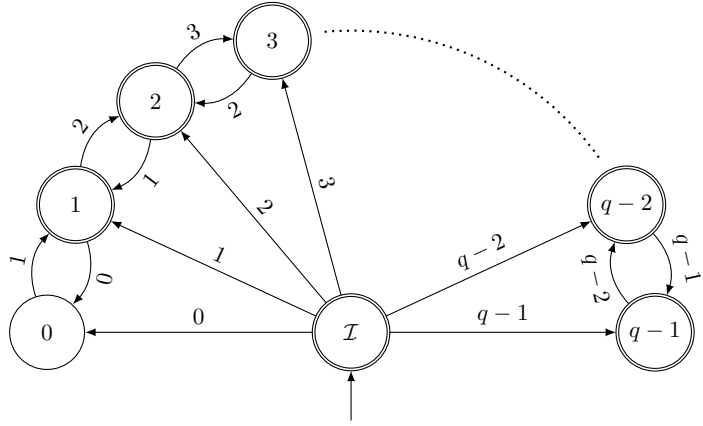
\begin{figure}
  \centering

  \begin{tikzpicture}[auto,
    initial text=, initial distance=5ex,
    >=latex,
    accepting text=,
    every state/.style={minimum size=3.2em},
    scale=0.8,
    every node/.style={scale=0.8}]

    \node[state, initial below, accepting] (I) at (0.000000, 0.000000) {$\mathcal{I}$};

    \node[state] (e0) at (180:5) {$0$};
    \node[state, accepting] (e1) at (155:5) {$1$};
    \node[state, accepting] (e2) at (130:5) {$2$};
    \node[state, accepting] (e3) at (105:5) {$3$};

    \draw[dotted, thick] (95:5) arc (95:35:5);

    \node[state, accepting] (eq2) at (25:5) {$q-2$};
    \node[state, accepting] (eq1) at (0:5) {$q-1$};

    \path[->] (I) edge node[rotate=0, anchor=south] {$0$} (e0);
    \path[->] (I) edge node[rotate=-25, anchor=south] {$1$} (e1);
    \path[->] (I) edge node[rotate=-50, anchor=south] {$2$} (e2);
    \path[->] (I) edge node[rotate=-75, anchor=south] {$3$} (e3);
    \path[->] (I) edge node[rotate=25, anchor=south] {$q-2$} (eq2);
    \path[->] (I) edge node[rotate=0, anchor=south] {$q-1$} (eq1);

    \path[->] (e0) edge[bend left] node[rotate=77.5, anchor=south] {$1$} (e1);
    \path[->] (e1) edge[bend left] node[rotate=77.5, anchor=north] {$0$} (e0);
    \path[->] (e1) edge[bend left] node[rotate=52.5, anchor=south] {$2$} (e2);
    \path[->] (e2) edge[bend left] node[rotate=52.5, anchor=north] {$1$} (e1);
    \path[->] (e2) edge[bend left] node[rotate=27.5, anchor=south] {$3$} (e3);
    \path[->] (e3) edge[bend left] node[rotate=27.5, anchor=north] {$2$} (e2);

    \path[->] (eq2) edge[bend left] node[rotate=-77.5, anchor=south] {$q-1$} (eq1);
    \path[->] (eq1) edge[bend left] node[rotate=-77.5, anchor=north] {$q-2$} (eq2);

  \end{tikzpicture}
  \vspace*{-1em}

  \caption{Automaton~$\mathcal{A}$ recognizing esthetic numbers.}
  \label{fig:esthetic-automaton}
\end{figure}

The language consisting of the $q$-ary digit expansions (seen as words
of digits) which are $q$-esthetic
is a regular language, because it is recognized by the
automaton~$\mathcal{A}$ in
Figure~\ref{fig:esthetic-automaton}. Therefore, the indicator sequence
of this language, i.e., the $n$th entry is $1$ if $n$ is $q$-esthetic
and $0$ otherwise is a $q$-automatic sequence and therefore also
$q$-regular. Let us name this sequence~$x(n)$.

Let $A_0$, \dots, $A_{q-1}$ be the transition matrices of the
automaton~$\mathcal{A}$, i.e., $A_r$ is the adjacency matrix of the
directed graph induced by a transition with digit~$r$.
To make this more explicit, we have
the following $(q+1)$-dimensional square
matrices: Each row and column corresponds to the states~$0$, $1$,
\dots, $q-1$, $\mathcal{I}$. In matrix~$A_r$, the only nonzero entries
are in column~$r\in\set{0,1,\dots,q-1}$, namely $1$ in the rows~$r-1$ and $r+1$ (if
available) and in row~$\mathcal{I}$ as there are transitions from
these states to state~$r$ in the automaton~$\mathcal{A}$.

Let us make this more concrete by considering $q=4$. We obtain the matrices
\begin{align*}
  A_0 &=
   \scalebox{0.8}{$\begin{pmatrix}
    0 & 0 & 0 & 0 & 0 \\
    1 & 0 & 0 & 0 & 0 \\
    0 & 0 & 0 & 0 & 0 \\
    0 & 0 & 0 & 0 & 0 \\
    1 & 0 & 0 & 0 & 0
  \end{pmatrix}$},
  &
  A_1 &=
  \scalebox{0.8}{$\begin{pmatrix}
    0 & 1 & 0 & 0 & 0 \\
    0 & 0 & 0 & 0 & 0 \\
    0 & 1 & 0 & 0 & 0 \\
    0 & 0 & 0 & 0 & 0 \\
    0 & 1 & 0 & 0 & 0
  \end{pmatrix}$},
  &
  A_2 &=
  \scalebox{0.8}{$\begin{pmatrix}
    0 & 0 & 0 & 0 & 0 \\
    0 & 0 & 1 & 0 & 0 \\
    0 & 0 & 0 & 0 & 0 \\
    0 & 0 & 1 & 0 & 0 \\
    0 & 0 & 1 & 0 & 0
  \end{pmatrix}$},
  &
  A_3 &=
  \scalebox{0.8}{$\begin{pmatrix}
    0 & 0 & 0 & 0 & 0 \\
    0 & 0 & 0 & 0 & 0 \\
    0 & 0 & 0 & 1 & 0 \\
    0 & 0 & 0 & 0 & 0 \\
    0 & 0 & 0 & 1 & 0
  \end{pmatrix}$}.
\end{align*}

We are almost at a $q$-linear representation of our sequence; we still
need vectors on both sides of the matrix products. We have
\begin{equation*}
  x(n) = e_{q+1}\, A_{r_0} \cdots A_{r_{\ell-1}} v(0)
\end{equation*}
for $r_{\ell-1} \dots r_0$ being the $q$-ary expansion of~$n$ and chosen vectors
$e_{q+1}=\begin{pmatrix}0& \dotsc& 0&1\end{pmatrix}$ and
$v(0)=\begin{pmatrix}0&1& \dotsc& 1\end{pmatrix}^\top$.
Strictly speaking, this is not yet a regular sequence: in the case of regular
sequences, we always have that $A_0 v(0)=v(0)$ which is not the case here. This
does not matter in this case: the difference leads to an additional constant in
the asymptotic analysis which is absorbed by the error term anyway.

To see that the above holds, we have two different interpretations:
The first is that the row vector
$w(n) = e_{q+1}\, A_{r_0} \cdots A_{r_{\ell-1}}$
is the unit vector corresponding to the most significant digit
of the $q$-ary expansion of~$n$ or, in view of the
automaton~$\mathcal{A}$, corresponding to the final state.
Note that we read the digit expansion from the least significant digit
to the most significant one
(although it would be possible the other way round as well).
We have $w(0)=e_{q+1}$
which corresponds to the empty word and
being in the initial state~$\mathcal{I}$ in the automaton.
The vector~$v(0)$ corresponds to the fact that
all states of~$\mathcal{A}$ except~$0$ are accepting.

The other interpretation is: The $r$th component of the column vector
$v(n) = A_{r_0} \cdots A_{r_{\ell-1}} v(0)$
has the following two meanings:
\begin{itemize}
\item In the automaton~$\mathcal{A}$, we start in state $r$ and then
  read the digit expansion of $n$. The $r$th component is then the indicator
  function whether we remain esthetic, i.e., end in an accepting
  state.
\item To a word ending with $r$ we append the digit expansion of
  $n$. The $r$th component is then the indicator function whether the result
  is an esthetic word.
\end{itemize}

At first glance, our problem here seems to be a special case of the
transducers studied in~\cite{Heuberger-Kropf-Prodinger:2015:output}. However, the
automaton~$\mathcal{A}$ is not complete. Adding a sink to have a
formally complete automaton, however, adds an eigenvalue $q$ and thus
a much larger dominant asymptotic term, which would then be multiplied
by~$0$. Therefore, the results
of~\cite{Heuberger-Kropf-Prodinger:2015:output} do not apply to this
case here.

\subsection{Full Asymptotics}

We now formulate our main result for the amount of esthetic numbers
smaller than a given integer~$N$. We abbreviate this amount by
$X(N) = \sum_{0 \le n < N} x(n)$
and have the following theorem.

\begin{theorem}
  \label{theorem:esthetic:asy}
  The number~$X(N)$ of $q$-esthetic numbers smaller than $N$ is
  \begin{multline}\label{eq:esthetic:asy-main}
    X(N) = \sum_{j\in\set{1,2,\dots,\ceil{\frac{q-2}{3}}}}
    N^{\log_q (2\cos(j\pi/(q+1)))} \Phi_{qj}(2\fractional{\log_{q^2} N})
    + \Oh[\big]{(\log N)^{\iverson{q \equiv -1 \tpmod 3}}}
  \end{multline}
  with $2$-periodic continuous functions~$\Phi_{qj}$.
  Moreover, we can effectively compute the Fourier coefficients of
  each~$\Phi_{qj}$.
  If $q$ is even, then the functions $\Phi_{qj}$ are actually $1$-periodic.
\end{theorem}

If $q=2$, then the theorem results in $X(N)=\Oh{\log N}$.
However, for each length, the only word of digits satisfying the
esthetic number condition has alternating digits $0$ and $1$,
starting with~$1$ at its most significant digit. The
corresponding numbers~$n$ are the
sequence~\href{https://oeis.org/A000975}{A000975}
(``Lichtenberg sequence'') in The On-Line
Encyclopedia of Integer Sequences~\cite{OEIS:2018}.

Back to a general~$q$: For the asymptotics,
the main quantities influencing the growth will turn out to be the
eigenvalues of the matrix~$C = A_0+\dots+A_{q-1}$. Continuing our
example $q=4$ above, this matrix is
\begin{equation*}
  C = A_0 + A_1 + A_2 + A_3 =
  \scalebox{0.8}{$\begin{pmatrix}
    0 & 1 & 0 & 0 & 0 \\
    1 & 0 & 1 & 0 & 0 \\
    0 & 1 & 0 & 1 & 0 \\
    0 & 0 & 1 & 0 & 0 \\
    1 & 1 & 1 & 1 & 0
  \end{pmatrix}$},
\end{equation*}
and its eigenvalues are
$\pm 2\cos(\frac{\pi}{5})=\pm \frac12\bigl(\sqrt{5} + 1\bigr) = \pm1.618\dots$,
$\pm 2\cos(\frac{2\pi}{5})=\pm \frac12\bigl(\sqrt{5} - 1\bigr) = \pm0.618\dots$
and $0$, all with algebraic and geometric multiplicity $1$. Therefore it turns out that
the growth of the main term is
$N^{\log_4(\sqrt{5} + 1) - \frac12}=N^{0.347\dots}$, see
Figure~\ref{fig:fluct-esthetic}

\begin{figure}
  \centering
  \includegraphics{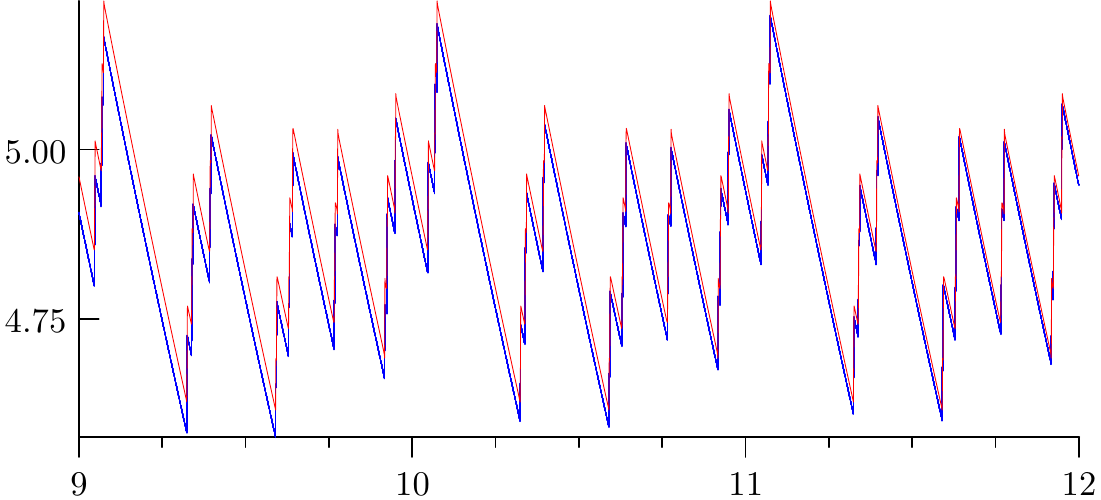}
  \caption{Fluctuation in the main term of the asymptotic expansion of $X(N)$
    for $q=4$.
    The figure shows $\f{\Phi_1}{u}$ (red) approximated by
    its trigonometric polynomial of degree~$1999$ as well as
    $X(4^u) / N^{u(\log_4(\sqrt{5} + 1) - \frac12)}$ (blue).}
  \label{fig:fluct-esthetic}
\end{figure}

The proof of Theorem~\ref{theorem:esthetic:asy} can be found in
Appendix~\ref{sec:proof-esthetic}.

\def\moveestheticproof{
\begin{proof}[Proof of Theorem~\ref{theorem:esthetic:asy}]
  We work out the conditions and parameters for using
  Theorem~\ref{theorem:simple}.

  \proofparagraph{Joint Spectral Radius}
  As all the square matrices $A_0$, \dots, $A_{q-1}$ have a maximum
  absolute row sum norm equal to $1$, the joint spectral radius of
  these matrices is bounded by~$1$.

  Let $r\in\set{1,\dots,q-1}$. Then any product with alternating
  factors $A_{r-1}$ and $A_r$, i.e., a finite product
  $A_{r-1}A_rA_{r-1}\cdots$, has absolute row sum norm at least~$1$ as
  the word $(r-1)r(r-1)\dots$ is $q$-esthetic. Therefore the joint
  spectral radius of $A_{r-1}$ and $A_r$ is at
  least~$1$. Consequently, the joint spectral radius of $A_0$, \dots,
  $A_{q-1}$ equals~$1$.

  \proofparagraph{Eigenvalues}
  The matrix $C = A_0+\dots+A_{q-1}$ has a block decomposition into
  \begin{equation*}
    C = 
    \left(\begin{array}{c|c}
      M & \mathbf{0} \\
      \hline
      \mathbf{1} & 0
    \end{array}\right)
  \end{equation*}
  for vectors~$\mathbf{0}$ (vector of zeros) and $\mathbf{1}$
  (vector of ones) of suitable dimension.
  Therefore, one eigenvalue of~$C$ is~$0$ and the others are the eigenvalues
  of~$M$ which are the zeros
  \begin{equation*}
    \lambda_j = 2 \cos\biggl(\frac{j\pi}{q+1}\biggr)
    \quad\text{for $j\in\set{1,\dots,q}$}
  \end{equation*}
  of the polynomials $p_q(x)$ which are recursively defined by $p_0(x)=1$,
  $p_1(x)=x$ and $p_\ell(x) = x p_{\ell-1}(x) - p_{\ell-2}(x)$ for $\ell\geq2$;
  see \cite[Sections~4 and~5]{Koninck-Doyon:2009:esthetic-numbers}.
  Note that up to replacing $x$ by $2x$, these polynomials~$p_\ell$
  are the Chebyshev polynomials of the second kind. This is not
  surprising: Chebyshev polynomials are frequently occurring phenomena
  in lattice path analysis, and we have such a lattice path here.

  It can be shown that in the case of even $q$, the vector $e_{q+1}$ lies in
  the sum of the left eigenspaces to the eigenvalues $2\cos(\frac{j\pi}{q+1})$ for
  \emph{odd} $j \in \{1, \ldots , q\}$ only. Therefore, the other eigenvalues
  can be omitted and the functions $\Phi_{qj}$ are actually $1$-periodic.

  \proofparagraph{Asymptotics}
  We apply our Theorem~\ref{theorem:simple}.
  We have $\lambda_j=-\lambda_{q+1-j}$, so we combine our approach
  with Proposition~\ref{proposition:symmetric-eigenvalues}. Moreover,
  we have $\lambda_j>1$ iff $\frac{j}{q+1}<\frac{1}{3}$ iff
  $j\leq\ceil{\frac{q-2}{3}}$. This results
  in~\eqref{eq:esthetic:asy-main}.

  \proofparagraph{Fourier Coefficients}
  We can compute the Fourier coefficients according to
  Theorem~\ref{theorem:simple} and
  Proposition~\ref{proposition:symmetric-eigenvalues}.
\end{proof}
}


\section{Asymptotics of Summatory Functions}
\label{sec:asy-summatory}

\subsection{Main Result on the Asymptotics}\label{section:introduction:main-result}

We are interested in the asymptotic behaviour of the summatory function
$X(N)=\sum_{0\le n<N}x(n)$.

We choose any
vector norm $\norm{\,\cdot\,}$ on $\C^d$ and its induced matrix norm. We set $C\coloneqq
\sum_{r=0}^{q-1}A_r$. We choose $R>0$ such that $\norm{A_{r_1}\dotsm
  A_{r_\ell}}=\Oh{R^\ell}$ holds for all $\ell\ge 0$ and $0\le r_1, \dotsc,
r_{\ell}<q$. In other words, $R$ is an upper bound for the joint spectral
radius of $A_1$, \ldots, $A_{q-1}$.
The spectrum of $C$, i.e., the set of eigenvalues of $C$, is denoted by
$\sigma(C)$. For $\lambda\in\C$, let $m(\lambda)$ denote the size of the
largest Jordan block of $C$ associated with $\lambda$; in particular,
$m(\lambda)=0$ if $\lambda\notin\sigma(C)$.
Finally, we consider the Dirichlet series\footnote{
Note that the summatory function $X(N)$ contains the summand $x(0)$ but the Dirichlet series cannot.
This is because the choice of including $x(0)$ into $X(N)$ will lead to more consistent results.}
\begin{equation*}
  \calX(s) = \sum_{n\ge 1} n^{-s}x(n), \qquad
  \calV(s) = \sum_{n\ge 1} n^{-s}v(n)
\end{equation*}
where $v(n)$ is the vector valued sequence defined in \eqref{eq:linear-representation}.
Of course, $\calX(s)$ is the first component of $\calV(s)$.
The principal value of the complex logarithm is denoted by $\log$. The
fractional part of a real number $z$ is denoted by $\fractional{z}\coloneqq z-\floor{z}$.

\begin{theorem}\label{theorem:simple}
  With the notations above, we have
  \begin{multline}\label{eq:formula-X-n}
    X(N) = \sum_{\substack{\lambda\in\sigma(C)\\\abs{\lambda}>R}}N^{\log_q\lambda}
    \sum_{0\le k<m(\lambda)}(\log_q N)^k
    \Phi_{\lambda k}(\fractional{\log_q N}) \\
    + \Oh[\big]{N^{\log_q R}(\log N)^{\max\setm{m(\lambda)}{\abs{\lambda}=R}}}
  \end{multline}
  for suitable $1$-periodic continuous functions $\Phi_{\lambda k}$. If there
  are no eigenvalues $\lambda\in\sigma(C)$ with $\abs{\lambda}\le R$, the
  $O$-term can be omitted.

  For $\abs{\lambda}>R$ and $0\le k<m(\lambda)$, the function $\Phi_{\lambda k}$ is Hölder continuous with any exponent
  smaller than $\log_q(\abs{\lambda}/R)$.

  The Dirichlet series $\calV(s)$ converges absolutely and uniformly on compact
  subsets of the half plane $\Re
  s>\log_q R +1$ and can be continued to a meromorphic function on the half plane $\Re
  s>\log_q R$.
  It satisfies the functional equation
  \begin{equation}\label{eq:functional-equation-V}
    (I-q^{-s}C)\calV(s)= \sum_{n=1}^{q-1}n^{-s}v(n) +
    q^{-s}\sum_{r=0}^{q-1}A_r \sum_{k\ge
      1}\binom{-s}{k}\Bigl(\frac{r}{q}\Bigr)^k \calV(s+k)
  \end{equation}
  for $\Re s>\log_q R$. The right side converges absolutely and uniformly on
  compact subsets of $\Re s>\log_q R$. In particular, $\calV(s)$ can only have
  poles where $q^s\in\sigma(C)$.

  For $\lambda\in\sigma(C)$ with
  $\abs{\lambda}>R$, the Fourier series
  \begin{equation*}
    \Phi_{\lambda k}(u) = \sum_{\ell\in \Z}\varphi_{\lambda k\ell}\exp(2\ell\pi i u)
  \end{equation*}
  converges pointwise for $u\in\R$ where
  \begin{equation}\label{eq:Fourier-coefficient:simple}
    \varphi_{\lambda k\ell} = \frac{(\log q)^k}{k!}
    \Res[\bigg]{\frac{\bigl(x(0)+\calX(s)\bigr)
      \bigl(s-\log_q \lambda-\frac{2\ell\pi i}{\log q}\bigr)^k}{s}}%
    {s=\log_q \lambda+\frac{2\ell\pi i}{\log q}}
  \end{equation}
  for $\ell\in\Z$, $0\le k<m(\lambda)$.
\end{theorem}
The above theorem is almost the formulation found
in~\cite{Heuberger-Krenn-Prodinger:2018:pascal-rhombus}, but with the
important difference that the technical condition
$\abs{\lambda}>\max\set{R, 1/q}$ is replaced with the condition
$\abs{\lambda}>R$. The latter condition is inherent in the problem:
single summands $x(n)$ might be as large as $n^{\log_q R}$ and must
therefore be absorbed by the error term in any smooth asymptotic
formula for the summatory function.

\begin{proof}[Sketch of Proof of Theorem~\ref{theorem:simple}]
  Use the proof of the corresponding theorem
  in~\cite{Heuberger-Krenn-Prodinger:2018:pascal-rhombus}, but replace
  the pseudo-Tauberian argument by
  Theorem~\ref{theorem:pseudo-Tauber}.
\end{proof}

\subsection{Fourier Coefficients \& Mellin--Perron Summation}\label{sec:heuristic}

We give a heuristic and non-rigorous argument explaining why the formula
\eqref{eq:Fourier-coefficient:simple} for the Fourier coefficients is
expected; see also \cite{Drmota-Grabner:2010}.

By the Mellin--Perron summation formula of order $0$ (see, for example,
\cite[Theorem~2.1]{Flajolet-Grabner-Kirschenhofer-Prodinger:1994:mellin}),
we have
\begin{equation*}
  \sum_{1\le n<N}x(n) + \frac{x(N)}{2} = \frac1{2\pi i}\int_{\max\set{\log_q R + 2,1}
    -i\infty}^{\max\set{\log_q R + 2,1} +i\infty} \calX(s)\frac{N^s\,\dd s}{s}.
\end{equation*}
Shifting the line of integration to the left and collecting the
residues at the location of the poles of $\calX(s)$ claimed in
Theorem~\ref{theorem:simple} yields the Fourier series expansion.
However, we have \emph{no analytic justification} that this is
allowed, so we need to work around this issue by reducing the problem
to higher order Mellin--Perron summation; details are to be found
in~\cite{Heuberger-Krenn-Prodinger:2018:pascal-rhombus}. One key
ingredient to make tracks back to our original summation problem is a
pseudo-Tauberian theorem; see below.

\def\movesymmetric{
\section{Fluctuations of Symmetrically Arranged Eigenvalues}
\label{sec:symmetric}

In our main results, the occurring fluctuations are always
$1$-periodic functions. However, if eigenvalues of the sum of matrices
of the linear representation are
arranged in a symmetric way, then we can combine summands and get
fluctuations with longer periods. This is in particular true if all
vertices of a regular polygon (with center~$0$) are eigenvalues.

\begin{proposition}\label{proposition:symmetric-eigenvalues}
  Let $\lambda\in\C$ and $k\in\N_0$. For a $p\in\N_0$ denote by $U_p$ the
  set of $p$th roots of unity. Suppose for each $\zeta\in U_p$
  we have a continuous $1$-periodic function
  \begin{equation*}
    \Phi_{(\zeta\lambda)k}(u)
    = \sum_{\ell\in\Z}\varphi_{(\zeta\lambda)k\ell}\exp(2\ell\pi i u)
  \end{equation*}
  whose Fourier coefficients are
  \begin{equation*}
    \varphi_{(\zeta \lambda)k\ell}
    =\Res[\bigg]{\calD(s)
      \Bigl(s - \log_q (\zeta\lambda) - \frac{2\ell\pi i}{\log q}\Bigr)^k}%
    {s=\log_q (\zeta\lambda) + \frac{2\ell\pi i}{\log q}}
  \end{equation*}
  for a suitable function $\calD(s)$.

  Then
  \begin{equation*}
    \sum_{\zeta\in U_p} N^{\log_q (\zeta\lambda)} (\log_q N)^k
    \Phi_{(\zeta\lambda)k}(\fractional{\log_q N})
    = N^{\log_q \lambda} (\log_q N)^k \Phi(p\fractional{\log_{q^p} N})
  \end{equation*}
  with a continuous $p$-periodic function
  \begin{equation*}
    \Phi(u)
    = \sum_{\ell\in\Z}\varphi_{\ell}\fexp[\Big]{\frac{2\ell\pi i}{p} u}
  \end{equation*}
  whose Fourier coefficients are
  \begin{equation*}
    \varphi_{\ell}
    =\Res[\bigg]{\calD(s)
      \Bigl(s - \log_q \lambda - \frac{2\ell\pi i}{p\log q}\Bigr)^k}%
    {s=\log_q \lambda + \frac{2\ell\pi i}{p\log q}}.
  \end{equation*}
\end{proposition}

Note that we again write $\Phi(p\fractional{\log_{q^p} N})$ to
optically emphasise the $p$-periodicity. Moreover, the factor
$(\log_q N)^k$ in the result could be cancelled, however it is there to
optically highlight the similarities to the main results (e.g.\@
Theorem~\ref{theorem:simple}).

In the case of a $q$-regular sequence which we analyse in this paper,
a different point of view is possible: The
sequence is $q^p$-regular as well
(by~\cite[Theorem~2.9]{Allouche-Shallit:1992:regular-sequences}) and
therefore, all eigenvalues $\zeta\lambda$ of the original sequence
become eigenvalues $\lambda^p$ whose algebraic multiplicity is the sum
of the individual multiplicities but the sizes of the corresponding
Jordan blocks do not change.
Moreover, the joint spectral radius is also taken to the $p$th power.
We apply, for example,
Theorem~\ref{theorem:simple} in our $q^p$-world and get again
$1$-period fluctuations.
Note that for actually computing the Fourier coefficients,
the approach presented in the proposition seems to be more suitable.

The above proposition will be used for
the analysis of esthetic numbers in
Section~\ref{sec:esthetic-numbers}.
The proof of Proposition~\ref{proposition:symmetric-eigenvalues}
can be found in
Appendix~\ref{sec:proof-symmetric-eigenvalues}.

}

\section{Pseudo-Tauberian Theorem}
\label{sec:pseudo-tauber}

In this section, we generalise a pseudo-Tau\-be\-rian argument by Flajolet, Grabner,
Kirschenhofer, Prodinger and
Tichy~\cite[Proposition~6.4]{Flajolet-Grabner-Kirschenhofer-Prodinger:1994:mellin}. In
contrast to their version, we allow for an additional logarithmic factor,
have weaker growth conditions on the Dirichlet series and
quantify the error. We also extend the result to all complex $\kappa$.

\begin{theorem}\label{theorem:pseudo-Tauber}
  Let $\kappa\in\C$ and $q>1$ be a real number,  $m$ be a
  positive integer, $\Phi_0$, \ldots, $\Phi_{m-1}$ be $1$-periodic Hölder continuous
  functions with exponent $\alpha>0$, and $0<\beta<\alpha$. Then there exist continuously differentiable functions
  $\Psi_{-1}$, $\Psi_{0}$, \ldots, $\Psi_{m-1}$, periodic with period $1$, and a constant $c$ such that
  \begin{multline}
    \sum_{1\le n< N}n^\kappa \sum_{\substack{j+k=m-1\\0\le j<m}}\frac{(\log n)^{k}}{k!}\Phi_j(\log_q n)\\
    =c + N^{\kappa+1}\sum_{\substack{k+j=m-1\\-1\le j<m}} \frac{(\log N)^{k}}{k!}\Psi_j(\log_q N)
    + \Oh[\big]{N^{\Re \kappa+1-\beta}}
    \label{eq:pseudo-Tauber-relation}
  \end{multline}
  for integers $N\to\infty$.

  Denote the Fourier coefficients of $\Phi_j$ and $\Psi_j$ by $\varphi_{j\ell}\coloneqq
  \int_0^1\Phi_j(u)\exp(-2\ell\pi i u)\, \dd u$ and $\psi_{j\ell}\coloneqq
  \int_0^1\Psi_j(u)\exp(-2\ell\pi i u)\, \dd u$, respectively.
  Then the corresponding generating functions fulfil
  \begin{equation}\label{eq:pseudo-Tauber-Fourier}
    \sum_{0\le j<m}\varphi_{j\ell}Z^j = \Bigl(\kappa+1+\frac{2\ell \pi i}{\log q} + Z\Bigr)\sum_{-1\le j<m}\psi_{j\ell}Z^j
    +\Oh{Z^m}
  \end{equation}
  for $\ell\in \Z$ and $Z\to 0$.

  If $q^{\kappa+1}\neq 1$, then $\Psi_{-1}$ vanishes.
\end{theorem}
\begin{remark}
  Note that the constant $c$ is absorbed by the error term if
  $\Re\kappa+1>\alpha$, in particular if $\Re\kappa>0$.
  Therefore, this constant does not occur in the
  article~\cite{Flajolet-Grabner-Kirschenhofer-Prodinger:1994:mellin}.
\end{remark}
\begin{remark}
  \label{remark:recurrence-fluctuation}
  The factor $\kappa+1+\frac{2\ell \pi i}{\log q} + Z$ in
  \eqref{eq:pseudo-Tauber-Fourier} will turn out to
  correspond exactly to the additional factor $s+1$ in the first order
  Mellin--Perron summation formula with the substitution
  $s=\kappa+\frac{2\ell\pi i}{\log q}+ Z$ such that the local expansion around
  the pole in $s=\kappa+\frac{2\ell\pi i}{\log q}$ of the Dirichlet generating
  function is conveniently written as a Laurent series in $Z$.
\end{remark}

\begin{proof}
  \proofparagraph{Notations}
  Without loss of generality, we assume that $q^{\Re \kappa+1}\neq q^{\alpha}$:
  otherwise, we slightly decrease $\alpha$ keeping the inequality
  $\beta<\alpha$ intact.
  We use the abbreviations $\Lambda\coloneqq \floor{\log_q N}$,
  $\nu\coloneqq \fractional{\log_q N}$, i.e.,
  $N=q^{\Lambda+\nu}$. We use the generating functions
  \begin{align*}
    \f{\Phi}{u, Z}&\coloneqq \sum_{0\le j<m}\Phi_j(u)Z^j,\\
    L(N, Z)&\coloneqq \sum_{1\le n<N}n^{\kappa+Z} \f{\Phi}{\log_q n, Z}=\sum_{1\le
             n<N}n^\kappa \fexp[\big]{(\log n) Z}\f{\Phi}{\log_q n, Z},\\
    Q(Z)&\coloneqq q^{\kappa+1+Z}
  \end{align*}
  for $0\le u\le 1$ and $0<\abs{Z}<2r$ where $r>0$ is chosen such that
  $r<(\alpha-\beta)/2$ and such that
  $Q(Z)\neq 1$ and $\abs{Q(Z)}\neq q^{\alpha}$ for these $Z$.
  (The condition $Z\neq 0$ is only needed for the case $q^{1+\kappa}=1$.)
  We will stick to the above choice of~$r$ and restrictions for~$Z$ throughout
  the proof.

  It is easily seen that  the left-hand side
  of~\eqref{eq:pseudo-Tauber-relation} equals $[Z^{m-1}]L(N, Z)$, where
  $[Z^{m-1}]$ denotes extraction of the coefficient of $Z^{m-1}$.

  \proofparagraph{Approximation of the Sum by an Integral}
  Splitting the range of summation with respect to powers of $q$ yields
  \begin{align*}
    L(N, Z) = \phantom{+\;}&
    \sum_{0\le p<\Lambda}\sum_{q^p\le n<q^{p+1}}n^{\kappa+Z}
    \f{\Phi}{\log_q n, Z} \\
    +\; &
    \sum_{q^\Lambda\le n< q^{\Lambda+\nu}}n^{\kappa+Z}\f{\Phi}{\log_q n, Z}.
  \end{align*}
  We write $n=q^px$ (or $n=q^\Lambda x$ for the second sum), use the
  periodicity of $\Phi$ in $u$ and get
  \begin{align*}
    L(N, Z) = \phantom{+\;}&
    \sum_{0\le p<\Lambda}Q(Z)^p\sum_{\substack{x\in q^{-p}\Z\\ 1\le x < q}}
     x^{\kappa+Z}\f{\Phi}{\log_q x, Z}\frac{1}{q^p} \\
    +\; &
    Q(Z)^\Lambda \sum_{\substack{x\in q^{-\Lambda}\Z\\ 1\le x < q^{\nu}}}
    x^{\kappa+Z}\f{\Phi}{\log_q x, Z}\frac{1}{q^\Lambda}.
  \end{align*}
  The inner sums are Riemann sums converging
  to the corresponding integrals for $p\to\infty$.
  We set
  \begin{equation*}
    I(u, Z)\coloneqq\int_{1}^{q^u}x^{\kappa+Z} \f{\Phi}{\log_q x, Z}\,\dd x.
  \end{equation*}
  It will be convenient to change variables $x=q^w$ in $I(u, Z)$ to get
  \begin{equation}\label{eq:Pseudo-Tauber:I-definition}
    I(u, Z)=(\log q)\int_{0}^{u}Q(Z)^w \f{\Phi}{w, Z}\,\dd w.
  \end{equation}
  We define the error~$\varepsilon_p(u, Z)$ by
  \begin{equation*}
    \sum_{\substack{x\in q^{-p}\Z\\
        1\le x < q^u}}x^{\kappa+Z} \f{\Phi}{\log_q x, Z}\frac1{q^p}=I(u, Z) +
    \varepsilon_{p}(u, Z).
  \end{equation*}
  As the sum and the integral are both analytic in $Z$, their difference
  $\varepsilon_p(u, Z)$ is analytic in $Z$, too.
  We bound~$\varepsilon_{p}(u, Z)$ by the difference of upper and lower
  Darboux sums (step size~$q^{-p}$)
  corresponding to the integral~$I(u, Z)$:
  On each interval of length $q^{-p}$, the maximum and minimum of a
  Hölder continuous function can differ by at most $\Oh{q^{-\alpha p}}$. As
  the integration interval as well as the range for $u$ and $Z$ are finite, this translates to the bound
  $\varepsilon_p(u, Z)=\Oh{q^{-\alpha p}}$ as $p\to\infty$
  uniformly in $0\le u\le 1$ and $\abs{Z}<2r$. This results in
    \begin{multline*}
    L(N, Z)=
    I(1, Z)\sum_{0\le p<\Lambda}Q(Z)^p
    + \sum_{0\le p<\Lambda}Q(Z)^p \varepsilon_{p}(1, Z)
    + I(\nu, Z)\,Q(Z)^{\Lambda} + Q(Z)^\Lambda  \varepsilon_{\Lambda}(\nu, Z).
  \end{multline*}
  If $\abs{Q(Z)}/q^\alpha=q^{\Re\kappa+1 + \Re Z -\alpha}<1$, i.e.,
  $\Re \kappa+\Re Z<\alpha-1$,
  the second sum involving the integration error
  converges absolutely and uniformly in $Z$
  for $\Lambda\to\infty$ to some analytic function
  $c'(Z)$; therefore, we can
  replace the second sum by
  $c'(Z)+\Oh[\big]{q^{(\Re \kappa+1+2r-\alpha)\Lambda}}=c'(Z)+\Oh[\big]{N^{\Re\kappa+1+2r-\alpha}}$
  in this case.
  If $\Re \kappa + \Re Z>\alpha-1$, then the second sum is
  $\Oh[\big]{q^{(\Re \kappa+2r+1-\alpha)\Lambda}}=\Oh[\big]{N^{\Re\kappa+1+2r-\alpha}}$.
  By our choice of $r$, the case $\Re \kappa+\Re Z=\alpha-1$ cannot occur.
  So in any case, we may write the second sum as
  $c'(Z)+\Oh[\big]{N^{\Re \kappa+1-\beta}}$ by our choice of $r$.
  The last summand involving $\varepsilon_{\Lambda}(\nu, Z)$ is absorbed by
  the error term of the second summand.
  Note that the error term is uniform in $Z$ and, by its construction,
  analytic in~$Z$.

  Thus we end up with
  \begin{equation}\label{eq:Pseudo-Tauber:L-decomposition}
    L(N, Z)= c'(Z) + S(N, Z) + \Oh[\big]{N^{\Re \kappa+1-\beta}}
  \end{equation}
  where
  \begin{equation}\label{eq:pseudo-Tauber-S-definition}
    S(N, Z)\coloneqq I(1, Z)\sum_{0\le
      p<\Lambda}Q(Z)^p+I(\nu, Z)Q(Z)^\Lambda.
  \end{equation}

  It remains to rewrite $S(N, Z)$ in the form required by
  \eqref{eq:pseudo-Tauber-relation}. We emphasise that we will compute $S(N, Z)$
  exactly, i.e., no more asymptotics for $N\to\infty$ will play any rôle.

  \proofparagraph{Construction of $\Psi$}
  We rewrite~\eqref{eq:pseudo-Tauber-S-definition} as
  \begin{align*}
    S(N, Z)&=
             I(1, Z)\frac{1-\f{Q}{Z}^\Lambda}{1-\f{Q}{Z}}
             + I(\nu, Z) \f{Q}{Z}^\Lambda.
  \end{align*}
  We replace
  $\Lambda$ by $\log_q N - \nu$ and use
  \begin{align*}
    \f{Q}{Z}^\Lambda
    &= \f{Q}{Z}^{\log_q N}\f{Q}{Z}^{-\nu}
    = N^{\kappa+1+Z}  \f{Q}{Z}^{-\nu}
  \end{align*}
  to get
  \begin{equation}\label{eq:Pseudo-Tauber:S-decomposition}
    S(N, Z)= \frac{I(1, Z)}{1-\f{Q}{Z}}
    + N^{\kappa+1+Z} \Psi(\nu, Z)
  \end{equation}
  with
  \begin{equation}\label{eq:Pseudo-Tauber:Psi-definition}
    \Psi(u, Z)\coloneqq  \f{Q}{Z}^{-u}
    \Bigl(I(u, Z)-\frac{I(1, Z)}{1-\f{Q}{Z}}\Bigr).
  \end{equation}

  \proofparagraph{Periodic Extension of $\Psi$}
  It is obvious that $\f{\Psi}{u, Z}$ is continuously differentiable in $u\in[0,
  1]$.
  We have
  \begin{equation*}
    \f{\Psi}{1, Z}=\frac{I(1, Z)}{\f{Q}{Z}}
    \Bigl(1-\frac{1}{1-\f{Q}{Z}}\Bigr)
    =-\frac{I(1, Z)}{1-\f{Q}{Z}}=\f{\Psi}{0, Z}
  \end{equation*}
  because $I(0, Z)=0$ by \eqref{eq:Pseudo-Tauber:I-definition}.
  The derivative of $\f{\Psi}{u, Z}$ with respect to $u$ is
  \begin{align*}
    \frac{\partial \f{\Psi}{u,Z}}{\partial u}
    &= -\bigl(\log\f{Q}{Z}\bigr) \f{\Psi}{u, Z}
    + (\log q) \f{Q}{Z}^{-u} \f{Q}{Z}^u \f{\Phi}{u, Z}\\
    &= -\bigl(\log\f{Q}{Z}\bigr) \f{\Psi}{u, Z} + (\log q) \f{\Phi}{u, Z},
  \end{align*}
  which implies that
  \begin{equation*}
    \frac{\partial \f{\Psi}{u,Z}}{\partial u}\Bigr|_{u=1}=\frac{\partial \f{\Psi}{u,Z}}{\partial u}\Bigr|_{u=0}.
  \end{equation*}
  We can therefore extend $\f{\Psi}{u, Z}$ to a $1$-periodic continuously
  differentiable function in $u$ on $\R$.

  \proofparagraph{Fourier Coefficients of $\Psi$}
  By using equations~\eqref{eq:Pseudo-Tauber:Psi-definition} and
  \eqref{eq:Pseudo-Tauber:I-definition}, $Q(Z)=q^{\kappa+1+Z}$, and
  $\exp(-2\ell\pi iu)=q^{-\chi_\ell u}$ with $\chi_\ell=\frac{2\pi i\ell}{\log q}$, we now express the Fourier coefficients of $\f{\Psi}{u, Z}$ in terms of those of
  $\f{\Phi}{u, Z}$ by
  \begin{multline*}
    \int_{0}^1 \f{\Psi}{u, Z} \exp(-2\ell\pi i u) \,\dd u\\
    \begin{aligned}
    &=
    (\log q)\int_{0\le w\le u\le 1}
      \f{Q}{Z}^{w-u} \f{\Phi}{w, Z} q^{-\chi_\ell u} \,\dd w\,\dd u \\
    &\phantom{=}\;
      -\frac{I(1, Z)}{1-\f{Q}{Z}} \int_0^1
        q^{-(\kappa+1+Z+\chi_\ell)u} \,\dd u\\
    &=
    (\log q)\int_{0\le w\le 1} \f{Q}{Z}^w \f{\Phi}{w, Z}
      \int_{w\le u\le 1} q^{-(\kappa+1+Z+\chi_\ell)u} \,\dd u\,\dd w \\
    &\phantom{=}\;
      -\frac{I(1, Z)}{(1-\f{Q}{Z})(\log q)(\kappa+1+Z+\chi_\ell)}
        \Bigl(1-\frac{1}{\f{Q}{Z}}\Bigr)\\
    &=
    \frac{1}{\kappa+1+Z+\chi_\ell}
      \int_0^1 \f{Q}{Z}^w \f{\Phi}{w, Z}
      \Bigl(q^{-(\kappa+1+Z+\chi_\ell)w}-\frac1{\f{Q}{Z}}\Bigr)
        \,\dd w \\
    &\phantom{=}\;
      + \frac{I(1, Z)}{\f{Q}{Z}(\log q)(\kappa+1+Z+\chi_\ell)}\\
    &=
    \frac{1}{\kappa+1+\chi_\ell+Z}
      \int_0^1 \f{\Phi}{w, Z} \fexp{-2\ell\pi i w} \,\dd w\\
    &\phantom{=}\;
      -\frac{1}{\f{Q}{Z} (\kappa+1+\chi_\ell+Z)}
        \int_0^1 \f{Q}{Z}^w \f{\Phi}{w, Z} \,\dd w\\
    &\phantom{=}\;
      + \frac{I(1, Z)}{\f{Q}{Z}(\log q)(\kappa+1+Z+\chi_\ell)}.
    \end{aligned}
  \end{multline*}
  The second and third summands cancel, and we get
  \begin{equation}\label{eq:Pseudo-Tauber:Fourier-Coefficients-GF}
    \Bigl(\kappa+1+\chi_\ell + Z\Bigr)
    \int_{0}^1 \f{\Psi}{u, Z}\exp(-2\ell\pi i u)\,\dd u =
    \int_0^1\f{\Phi}{w, Z}
    \exp(-2\ell\pi i w)\,\dd w.
  \end{equation}

  \proofparagraph{Extracting Coefficients}
  By~\eqref{eq:Pseudo-Tauber:Psi-definition}, $\f{\Psi}{u, Z}$ is analytic in $Z$
  for $0<\abs{Z}<2r$. If $q^{\kappa+1}\neq 1$, then it is analytic in $Z=0$, too. If
  $q^{\kappa+1}=1$, then~\eqref{eq:Pseudo-Tauber:Psi-definition} implies that $\f{\Psi}{u, Z}$
  might have a simple pole in $Z=0$.
  Note that all other possible poles have been excluded by our choice of $r$.
  For $j\ge -1$, we write
  \begin{equation*}
    \Psi_j(u)\coloneqq [Z^j]\f{\Psi}{u, Z}
  \end{equation*}
  and use Cauchy's formula to obtain
  \begin{equation*}
    \Psi_j(u) = \frac1{2\pi i}\oint_{\abs{Z}=r}\frac{\f{\Psi}{u, Z}}{Z^{j+1}}\,\dd Z.
  \end{equation*}
  This and the properties of $\f{\Psi}{u, Z}$ established above
  imply that $\Psi_j$ is a $1$-periodic continuously differentiable function.

  Inserting \eqref{eq:Pseudo-Tauber:S-decomposition}
  in~\eqref{eq:Pseudo-Tauber:L-decomposition} and extracting the coefficient of
  $Z^{m-1}$ using Cauchy's theorem and the analyticity of the error in $Z$ yields~\eqref{eq:pseudo-Tauber-relation}
  with $c=[Z^{m-1}]\bigl(c'(Z) + \frac{I(1, Z)}{1-\f{Q}{Z}}\bigr)$.
  Rewriting
  \eqref{eq:Pseudo-Tauber:Fourier-Coefficients-GF} in terms of $\Psi_j$ and $\Phi_j$ leads to~\eqref{eq:pseudo-Tauber-Fourier}.
  Note that we have to add $\Oh{Z^m}$ in~\eqref{eq:pseudo-Tauber-Fourier} to
  compensate the fact that we do not include $\psi_{j\ell}$ for $j\ge m$.
\end{proof}

\movesymmetric

\clearpage

\bibliography{bib/cheub}
\bibliographystyle{bibstyle/amsplainurl}

\clearpage
\appendix

\section{Proof of Proposition~\ref{proposition:symmetric-eigenvalues}}
\label{sec:proof-symmetric-eigenvalues}

\begin{proof}[Proof of Proposition~\ref{proposition:symmetric-eigenvalues}]
  We set
  \begin{equation*}
    j_0\coloneqq \floor[\bigg]{-\frac{p\bigl(\pi+\arg(\lambda)\bigr)}{2\pi}}+1
  \end{equation*}
  with the motive that
  \begin{equation*}
    -\pi<\arg(\lambda) + \frac{2j\pi}{p}\le \pi
  \end{equation*}
  holds for $j_0\le j<j_0+p$.
  This implies that for $j_0\le j<j_0+p$, the $p$th root of unity~$\zeta_j\coloneqq \exp(2j\pi i/p)$
  runs through the elements of $U_p$ such
  that $\log_q(\lambda \zeta_j)=\log_q(\lambda) +  2j\pi i/(p\log q)$.
  Then
  \begin{align*}
    N^{\log_q(\zeta_j\lambda)}
    &= N^{\log_q \lambda} \exp\Bigl(\frac{2j\pi i}{p}\log_q N\Bigr)\\
    &= N^{\log_q \lambda} \exp(2j\pi i\log_{q^p} N)
    = N^{\log_q \lambda} \exp(2j\pi i\fractional{\log_{q^p} N}).
  \end{align*}
  We set
  \begin{equation*}
    \Phi(u)\coloneqq \sum_{j_0\le j<j_0+p} \exp\Bigl(\frac{2j\pi i}{p}u\Bigr)\Phi_{(\zeta_j\lambda)k}(u),
  \end{equation*}
  thus $\Phi$ is a $p$-periodic function.

  For the Fourier series expansion, we get
  \begin{multline*}
    \Phi(u)=\sum_{\ell\in\Z} \sum_{j_0\le j<j_0+p}
    \Res[\bigg]{\calD(s)
      \Bigl(s - \log_q \lambda - \frac{2(\ell+\frac{j}{p})\pi i}{\log q}\Bigr)^k}%
    {s=\log_q \lambda + \frac{2(\ell+\frac{j}{p})\pi i}{\log q}} \\
    \times \f[\Big]{\exp}{2\pi i \Bigl(\ell+\frac{j}{p}\Bigr)u}
  \end{multline*}
  Replacing $\ell p+j$ by $\ell$ leads to the Fourier series claimed in the
  proposition.
\end{proof}

\section{Proof of Theorem~\ref{theorem:esthetic:asy}}
\label{sec:proof-esthetic}

\moveestheticproof

\end{document}